\documentclass[reqno,11pt]{amsart}
\usepackage[T1]{fontenc} 
\usepackage[utf8]{inputenc} 
\usepackage{comment}
\usepackage{mathtools}
\usepackage{amssymb}
\usepackage{amsthm}
\usepackage{bbm}
\usepackage{mathrsfs}
\usepackage[english]{babel} 
\usepackage{enumitem}
\usepackage{dsfont}
\usepackage{url}
\usepackage{hyperref}
\usepackage{tikz}

\hypersetup{linktocpage=true, colorlinks=true, breaklinks=true, linkcolor=blue, menucolor=black, urlcolor=black,citecolor=blue,filecolor=green}

\theoremstyle{plain}
\newtheorem{theorem}{Theorem}[section]
\newtheorem{proposition}[theorem]{Proposition}
\newtheorem{corollary}[theorem]{Corollary}
\newtheorem{lemma}[theorem]{Lemma}

\theoremstyle{definition}

\newtheorem{remark}[theorem]{Remark}

\newtheorem{assumption}[theorem]{Assumption}

\numberwithin{equation}{section}

\newcommand{\nn}{\nonumber}

\DeclareMathSymbol{\leqslant}{\mathalpha}{AMSa}{"36}
\DeclareMathSymbol{\geqslant}{\mathalpha}{AMSa}{"3E}
\DeclareMathSymbol{\doteqdot}{\mathalpha}{AMSa}{"2B}
\DeclareMathSymbol{\circlearrowright}{\mathalpha}{AMSa}{"08}
\DeclareMathSymbol{\subsetneq}{\mathalpha}{AMSb}{"28}
\DeclareMathSymbol{\supsetneq}{\mathalpha}{AMSb}{"29}
\renewcommand{\leq}{\;\leqslant\;}

\newcommand{\e}[1]{\,{\rm e}^{#1}\,}

\DeclareMathOperator*{\supp}{\text{supp}}

\newcommand{\upchi}{\raise 2pt \hbox{$\chi$}}

\newcommand{\caF}{{\mathcal F}}

\newcommand{\caN}{{\mathcal N}}

\newcommand{\caP}{{\mathcal P}}

\newcommand{\caS}{{\mathcal S}}

\newcommand{\bbE}{{\mathbb E}}

\newcommand{\bbN}{{\mathbb N}}

\newcommand{\bbP}{{\mathbb P}}

\newcommand{\bbR}{{\mathbb R}}

\newcommand{\bbZ}{{\mathbb Z}}

\newcommand{\de}{{\mathrm{d}}}
\newcommand{\one}{{\mathbbm{1}}}
\newcommand{\bbPh}{\hat{\bbP}}
\newcommand{\bbPw}{\Tilde{\bbP}}

\def\lt{\left}
\def\rt{\right}

\title{On Anti-Confinement Estimates for Self-Repelling Random Walks}
\author{Tobias Schmidt}
\address{Department of Mathematics, TU Darmstadt, Germany}
\email{tobias.schmidt@tu-darmstadt.de}

\author{Mark Sellke}
\address{Department of Statistics, Harvard University, Cambridge, MA, USA}
\email{msellke@fas.harvard.edu}
\date{\today}
\begin{document}
\maketitle

\begin{abstract}
    We study a class of $d$-dimensional random walks, including the two-dimensional simple random walk, reweighted by a self-repelling Gibbsian pair potential. We prove lower bounds on the diffusion
    constant for short-range interactions, and superdiffusive behavior in case the
    interaction is sufficiently long-range. Finally, we show that in the superdiffusive regime, faster temporal decay can be compensated by stronger spatial repulsion and vice-versa.
    Our technique combines GKS-based correlation inequalities on path space with recursive multi-scale estimates.
\end{abstract}

\section{Introduction}
The main object of interest in this work is a discrete-time random walk $\bbP$ in $d \in \bbN$ dimensions which is perturbed by a pair potential $W$. Let $x$ be a random walk trajectory, and let $(\phi_j)_{j \le T}$ be a family of i.i.d. one-step increments, whose law is given by $\eta$. Throughout, $\eta$ will be assumed to be a product law, meaning that $\phi = (\phi^1,...,\phi^d) \in \bbR^d$ has law
\[
\eta (\de \phi) = \prod \limits_{i=1} ^d \nu(\de \phi^i).
\]
Therefore, $\nu$ can be interpreted as the increment law of a $1$-dimensional random walk and $\eta$ is the law of a random walk where each coordinate evolves independently. In principle, we could also allow each dimension $j \le d$ to evolve according to a different increment law $\nu_j$, but choose it the same in each dimension for convenience.
We assume that $\nu$ is symmetric (i.e. $\nu(-\de \omega) = \nu (\de \omega)$) and that $\supp(\nu)$ is compact. We will refer to such random walks as product walks in all that follows.
Taking these definitions, it holds that
$$x = (x_0,x_1,...,x_T) \stackrel{d}{=} \Big(0,\phi_1, \phi_1 + \phi_2,...,\sum\limits_{j=1} ^T \phi_j \Big).$$
With this notation at hand, we are interested in properties of the probability measure
\begin{equation}
    \label{equ:perturbed_rw}
    \begin{split}
        \hat{\bbP}_{\alpha,T}(\de x) &=  \frac{1}{Z_{\alpha,T}} \exp \lt( \alpha \sum\limits_{0 \le i< j \le T} W( x_j - x_i , j-i ) \rt) \bbP(\de x) \\
        & = \frac{1}{Z_{\alpha,T}} \exp \lt( \alpha \sum\limits_{0 \le i< j \le T} W \lt( \sum\limits_{k = i+1} ^j \phi_k , j-i  \rt) \rt) \prod\limits_{k=1} ^T \eta (\de \phi_k).
    \end{split}
\end{equation}
These measures arise naturally when studying a (discrete) quantum particle coupled to a bosonic field. While the law in this case would be a continuous time random walk, we will be content with studying the discrete analogue in this paper. For intuition regarding the effect of $W$ on the walk, consider 
\[
W(z,t) = s(\Vert z \Vert) \rho(t).
\]
In the above $\rho: \bbN \rightarrow \bbR_{\ge 0}$ should be thought of as decreasing and $s : \bbR \rightarrow \bbR$ is either decreasing or increasing. $s$ \textit{decreasing} intuitively means that trajectories with small increments $x_j - x_i$ are favorable. On the other hand, $s$ \textit{increasing} is supposed to reward the path for not returning to previously seen locations. The time-scale at which this self-interaction is relevant is determined by the time-decay $\rho$ and the entire self-interaction is weighted by the coupling constant $\alpha > 0$. 

The bulk of the literature on measures of the form \eqref{equ:perturbed_rw} considers $W$ which are self-attracting; see for example \cite{BrySla95, Bolthausen1997, Bo02,OsSp99,BePo23,BaMuSeVa23,Se24,BeSchSe25,BeSchSe24}. Of interest are in particular the mean-square displacement $ \Vert x_T \Vert^2$ under $\bbPh_{\alpha,T}$, which is expected to depend inversely on the coupling constant $\alpha$ \cite{BePo23,BaMuSeVa23,Se24,BeSchSe25}. From a physical point of view, the asymptotic mean-square displacement is a quantity that describes how heavy a particle appears to be when coupled to a polarizable medium \cite{Sp87,DySp20}. 
In this work we instead take $W$ to be self-repelling, which intuitively leads to a mean-square displacement that increases with $\alpha$. 

Next, we take a closer look at the time-decay $\rho$. If $\rho(j) = 0$ for all $j \ge 2$
, then it is easily seen that $\bbPh_{\alpha,T} = \bbP$. If one has at least $\rho(2) > 0$, then the measure $\bbPh_{\alpha,T}$ introduces positive correlation between two neighboring random walk steps $\phi_i$ and $\phi_{i+1}$, which suffices to show that the diffusion constant depends positively on $\alpha$. On the other hand, consider a time decay like $\rho(t) \approx t^{-\xi}$ for some $\xi > 0$. If $\xi$ is sufficiently small, one is led to think that the random walk moves quite fast on large time-scales. It will be shown that long-range interactions can indeed lead to \textit{superdiffusive} behavior, meaning that the process has mean-square displacement which grows faster than linear in $T$.

Finally, let us note that on a heuristic level, there is monotonicity in the time-decay: consider once again an interaction of product form $W(z,t) = s(\Vert z \Vert) \rho(t)$. Let us take $\theta(t) \ge \rho(t)$ and set $W'(z,t) = s( \Vert z \Vert) \theta(t)$. It should follow that a measure reweighted by the interaction $W'$ diffuses more than a random walk reweighted by $W$. In a similar fashion, monotonic behavior in the spatial interaction is expected. While such monotonicity statements are typically difficult to verify rigorously, they can be established using the correlation inequalities introduced in this work; we refer to Lemma \ref{lemma:omitting_interactions} for a more general formulation. For a related monotonicity statement, see also Remark \ref{rem:density_replacement}.

One can also ask what happens if $s$ is replaced by a weaker spatial interaction, but $\rho$ is also made weaker in the sense that longer interactions matter more. Following this line of thought, one is led to believe that a time-decay resulting in superdiffusive behavior depends on the spatial interaction in the sense that ``stronger'' spatial repulsion allows for faster time-decay. We will show this for the special class of $W$ that we consider; see Theorem \ref{thm:long_range_superdiffusive_behavior} for the precise statement.

Let us now introduce the class of $W$ that we can treat. In what follows, it will be convenient to interpret the steps of the random walk as \textit{spins} on the lattice $\bbN$. The main requirement for all $W$ that we study is that interactions between two spins $\phi_m$ and $\phi_n$ are sums of products of powers of their coordinates. I.e., terms of the form 
\[
\sum_{i,j} \sum\limits_{1\leq k,l \leq d} c_{i,j} (\phi_m ^k)^i (\phi_n ^l)^j
\]
with $c_{i,j} \ge 0$ determine the correlation between two spin--coordinates, which allows us to apply correlation inequalities first introduced in \cite{gi70}. The \textit{GKS inequality} that will be used heavily in what follows is a slight reinterpretation of a well-known correlation inequality (e.g. \cite[Theorem 12.1]{simon79}), which has already been applied in similar works (e.g. \cite{OsSp99}). For a brief overview of other applications we refer to \cite[Chapter 12]{simon79}. Let us finally state our technical assumptions on the pair potential $W$.

\begin{assumption}
    \label{ass:short_range}
    For all $t \in \bbN$
    and for $z= (z^1,...,z^d) \in \bbR^d$ 
    it holds that
    $$W(z,t) = \sum\limits_{i=1} ^{n(t)} c_{i,t} \lt( \sum\limits_{p=1} ^d (z^p) ^q \rt)^i,$$
    where $c_{i,t} \ge 0$ for all $1 \le i \le n(t)$, and where $q \in \bbN$ is fixed.
    Moreover, $c_{i,t} > 0$ for some $i \in \bbN$ and $t = 2$.
\end{assumption}
Specializing to $q=2$ and $n(t) \equiv 1$ recovers the interaction
$W(z,t) = c_{1,t} \Vert z \Vert ^2$. Similarly, by choosing $q=2$ and $i$ arbitrary, even polynomials as a function of the euclidean norm are included in the interactions we allow. 
We also want to note that a similar class of functions has been considered in \cite{OsSp99}. In their work, however, Osada and Spohn consider negative coupling constants, which leads to a self-attracting pair potential. Moreover, they also consider perturbed Brownian motions instead of (simple) random walks. 

\begin{remark}
    One interesting example of $\bbP$ that we can treat is the two--dimensional simple random walk, at least if $q=2$. While this random walk is not naturally of product form, we recall that a rotation of space by $45$ degrees yields two i.i.d. processes. 
    Similar arguments were carried out in \cite{PrSaSc08}. As in our setup, their result should be extendable to $d \ge 3$, but it is non-obvious how to do this, since in higher dimension, the simple random walk is not (a rotation of) a product walk anymore.
\end{remark}

\begin{remark}
     In case the base measure is a random walk with fixed step size $a > 0$ (which we denote by $\bbP_a$), we can also treat negative $c_{i,t}$ but require that every $p$-spin interaction (meaning the product of $p$ distinct spins) has non-negative sign. In other words, after multiplying out all polynomials, the pair potential must give rise to interactions of the form
     $$\sum\limits_{I \subseteq [T]} c_I \prod\limits_{i \in I} \prod\limits_{j=1} ^d (\phi_i ^j)  ^{ k_{i,j}}$$
     with $c_I \ge 0$ for all $I \subseteq [T]$ and $k_{i,j} \in \bbN$.
     Consider e.g. $W(z,t) = \rho(t) ( \delta z^2 + \beta z^4)$ with $\rho \ge 0$. By writing $\bbP_a \lt((x_n)_{n \ge 0} \rt) \stackrel{d}{=} \bbP_1 \lt( (ax_n)_{n \ge 0} \rt)$, we find that the interaction changes to  
    $$W(z,t) = \rho(t) \lt( a^2 \delta \Vert z \Vert^2 + \beta a^4 \Vert z\Vert^4 \rt)$$
    with respect to the simple random walk.
    Multiplying out all polynomials shows that $\beta a^2 > \delta$ is sufficient to guarantee that all $p$-spin interactions have non-negative sign, which is what is required to apply the GKS inequality. In that sense, Assumption \ref{ass:short_range} can be weakened to a broader class of functions in the space coordinate.
\end{remark}
We now state our first result for (potentially) short-range interactions. 
\begin{theorem}
    \label{thm:short_range}
    Let $\bbPh_{\alpha,T}$ obey Assumption \ref{ass:short_range}. Assume that there exists an even $i\in\bbN$ such that $c_{i,2}>0$. Set $\zeta := c_{i,2}$. Then, there exist constants $C>0$, $a >0$ such that for all $\alpha \ge 1$
    $$\limsup\limits_{T \to \infty} \frac {1} {dT} \bbE^{\bbPh_{\alpha,T}} [ \Vert x_T \Vert^2] \ge C \e{a\zeta\alpha}.$$
\end{theorem} 

\begin{remark}
    \label{rem:s.r.w.suffices}
    The assumption that $c_{i,2} > 0$ for some $i \in \bbN$ is not a restriction. Let us assume that there exists a natural number $j > 2$ so that $c_{i,j} > 0$. Define the spins
    $$\Tilde{\phi}_l := \begin{cases}
        \sum\limits_{k= j \lfloor l/2 \rfloor +1  } ^{j \lfloor l/2 \rfloor  + \lceil j/2 \rceil} \phi_k \text{ if $l$ odd,} \\
        \sum\limits_{k= j \lfloor (l-1)/2 \rfloor + 1 + \lceil j/2 \rceil  } ^{ j \lfloor (l-1)/2 \rfloor  + j } \phi_k \text{ if $l$ even.}
    \end{cases}$$
    This formula regroups the spins on an interval of length $j$ into two disjoint sums. Consider for simplicity the first interval  $[0,j]$. Then $\Tilde{\phi}_1 = \sum_{k=1} ^{\lceil j/2 \rceil} \phi_k$, $\Tilde{\phi}_2 = \sum_{k=1 + \lceil j/2 \rceil} ^{j} \phi_k$ \footnote{One requires this slightly unhandy formula because  $j$ can be odd. If $j$ is even, then we simply split the spins $1,...,j$ into two groups of even size, namely $1,...,j/2$ and $j/2 +1,...,j$.}. 
    It is then possible to apply Wells' inequality (see Theorem \ref{thm:wells_inequality}), which allows us to renormalize $\Tilde{\phi_l}$ to $a X_l$, where $a > 0$ and $X_l$ is Rademacher distributed. In other words, it suffices to treat the case $t=2$.
\end{remark}

The previous statement is only interesting if even polynomials have non-zero coefficients. If an odd-power interaction has weight greater than $0$, the following can be deduced.
\begin{proposition}
    \label{prop:odd_interaction}
    Let $\bbPh_{\alpha,T}$ obey Assumption \ref{ass:short_range}. Assume that $q$ is odd and that there exists an odd $\ell\in\bbN$ with $c_{\ell,2}>0$.
    Then, for every $T \in \bbN$, some constant $c>0$ and $1 \le i \le d$
    $$\bbE^{\bbPh_{\alpha,T}} [ x_T ^i ] \ge cT.$$
    In particular, the process is ballistic, meaning 
    $$\bbE^{\bbPh_{\alpha,T}} [ \Vert x_T \Vert ^2 ] \ge c^2 T^2.$$
    
\end{proposition}
The previous result is not surprising, because odd-power interactions are genuinely increasing and not symmetric, meaning they induce a tilt for each spin separately. In other words, the expectation of single spin sites can be verified to be greater than $0$ in this case.

As previously alluded to, our second result concerns long-range interactions. We are interested in finding a regime of superdiffusive behavior. In order to do this, we will study the following path measures.
Define for $\alpha,\gamma,\xi > 0$ the measure
\begin{equation}
    \label{equ:long_range_rw}
    \begin{split}
        \hat{\bbP}_{\alpha,T,\gamma,\xi}(\de x) &=  \frac{1}{Z_{\alpha,T,\gamma,\xi}} \exp \lt( \alpha \sum\limits_{0 \le i< j \le T} \frac{ \Vert x_j - x_i \Vert^\gamma}{|j-i|^{\xi}} \rt) \bbP(\de x) .
    \end{split}
\end{equation}

\begin{theorem}
    \label{thm:long_range_superdiffusive_behavior}
    Let $\gamma \in 2\bbN$ and let $c_{\mathrm{crit}}= \log_2(1+\tanh(2))$. Take $\xi = \gamma/2 + c$ with $c \in (0,c_{\mathrm{crit}})$. Then, with $\alpha_* (c) = 2^c \mathrm{arctanh}(2^c - 1)$ and some constant $C>0$ it holds
    for all $\alpha \ge \alpha_* (c)$ that
    $$\bbE^{\bbPh_{\alpha,T,\gamma,\xi}}[ \Vert x_T \Vert ^2] \ge C T^{1+ \log_2(1+ \tanh(2))}.$$
\end{theorem}
Put in words, we show superdiffusive behavior for fixed $\gamma$ and $\xi = \gamma/2 + c$, where $c <c_{\mathrm{crit}}\approx 0.973$. Setting $\gamma =2$ and considering the simple random walk in $d=1$ resembles a model similar to Dyson's hierarchical model at inverse temperature $\alpha$ with time decay $\approx 1/ |i-j|^{\xi-2}$ \cite{Dyson1969,PrSaSc08}. In this model it is known that long--range order exists for a time decay $\xi < 4$ and $\alpha$ large enough. We therefore expect that the restriction to $c < \log_2 ( 1+\tanh(2))$ instead of $c \le 3$ to be an artifact of our approach. One can nevertheless ask for the correct scaling of $\xi$ in terms of $\gamma$. Here, the following heuristic can be helpful.

Let us consider the case $d=1$, $q=1$ and write
\[
    x_n=\sum_{m=1}^n \phi_m,\qquad \phi_m\in\{\pm1\},
\]
so that for  $t\ge1$,
\[
    x_{i+t}-x_i = \sum_{m=i+1}^{i+t}\phi_m =: S_{i,t}.
\]
For the exponent $\gamma=2k$ the self--interaction of the path is therefore built out of the terms $S_{i,t}^{2k}$. We now consider only $t > k$, because long--range interactions are expected to determine the behavior of the system.
Expanding $S_{0,t}^{2k}=(\sum_{m=1}^t \phi_m)^{2k}$ and using $\phi_m^2=1$, the coefficient of a fixed two--body interaction $\phi_a\phi_b$ ($a\neq b$) is obtained by counting monomials in which $a$ and $b$ appear an odd number of times and all other indices appear an even number of times.  The leading contribution comes from choosing $k-1$ additional distinct indices each occurring twice; basic combinatorics
shows that the pair coefficient scales as $t^{k-1}$, meaning
\begin{equation}
    \label{eq:pair-coeff-heuristic}
    S_{0,t}^{2k}
    \;=\; A_k(t)
    \;+\; C_k\, t^{k-1}\!\sum_{1\le a<b\le t}\phi_a\phi_b
    \;+\; \text{(higher--body spin terms)}
    \;+\; O(t^{k-2})\!\sum_{a<b}\phi_a\phi_b,
\end{equation}
for constants $C_k>0$ (and some scalar term $A_k(t)$) independent of $t$.

Fix two spins $\phi_m,\phi_n$ at distance $r =|m-n|$. For each $t\ge r$, the number of blocks of length $t$ that contain both $m$ and $n$ is asymptotically $(t-r)$; hence, the induced coupling at distance $r$ behaves as
\begin{equation}
    \label{eq:Jr-heuristic}
    \sum_{t\ge r} (t-r)\,\frac{t^{k-1}}{t^\xi}
    \;\asymp\;
    \sum_{t\ge r} t^{k-\xi}
    \;\asymp\;
    r^{-(\xi-k-1)} ,
\end{equation}
whenever $\xi>k+1$.
Thus, at the level of two body interactions, the system behaves like a one--dimensional long--range ferromagnet with tail $\sim r^{-s}$, where $s= \xi-\frac{\gamma}{2}-1$. In the quadratic case $\gamma=2$ this reproduces the well-known Dyson model and yields the expected critical threshold $\xi_c=4$ \cite{PrSaSc08}. One therefore expects that superdiffusive behavior occurs when the effective exponent satisfies $s\le2$, i.e.
$$ \xi < 3+ \frac \gamma 2$$
at sufficiently low temperature. For $\xi>\xi_c(\gamma)$ the induced long--range coupling is expected to be too weak to induce long--range alignment of spins, and diffusive behavior is expected (possibly with a renormalized diffusion constant).
Our approach can therefore verify the correct scaling of $\xi$ in terms of $\gamma$, but is slightly off on the exact threshold for which superdiffusivity is expected to occur. 

In order to understand why our results are most probably not tight, let us briefly sketch our approach. First, we may write
$$x_T = \sum\limits_{i=1} ^{T/2} \phi_i + \sum\limits_{i=T/2 +1} ^{T} \phi_i =: \sigma_\mathrm L + \sigma_\mathrm R.$$
Next, decompose the Gibbs interaction into two disjoint blocks and all interactions between those blocks, i.e.
\begin{equation}
    \nn
    \begin{split}
        \sum\limits_{0 \le i< j \le T} W( x_j - x_i , j-i ) &= \sum\limits_{0 \le i < j \le T/2} W( x_j - x_i , j-i ) + \sum\limits_{T/2 + 1 \le i < j \le T} W( x_j - x_i , j-i )  \\
        &+ \sum_{\substack{
0 \le i \le T/2 \\
T/2 + 1 \le j \le T
}}
 W( x_j - x_i , j-i ).
    \end{split}
\end{equation}
By GKS, we may replace 
the third term by $W(x_T,T)$, omitting all other interactions between the left and the right block. 
 Specializing to $W(z,T) = \Vert z \Vert^\gamma / T^{\gamma/2 + c}$ for $c>0$, one can deduce that $W(x_T,T)$ is more strongly repulsive (in the GKS sense) than $(\sigma_\mathrm L + \sigma_\mathrm R) ^2 /T^{1+c}$ (see Lemma \ref{lemma:gamma_two_suffices}). 
 By ignoring all other interactions between the left block and the right block, $\sigma_\mathrm L$ and $\sigma_\mathrm R$ are independent besides the single interaction between them. This makes it tractable on how much the Gibbs interaction $(\sigma_\mathrm L + \sigma_\mathrm R) ^2 /T^{1+c}$ improves the variance of $x_T$ compared to the case where $\sigma_\mathrm L$ and  $\sigma_\mathrm R$ are independent; we refer to Lemma \ref{lemma:AB-tanh_bound} for the details. The price to pay for this independence is that we ignore a lot of interactions between the blocks, which likely accounts for the discrepancy between the previously stated heuristics and our Theorem \ref{thm:long_range_superdiffusive_behavior}.

\section{Related Work}
The model we study in this paper is essentially a \emph{self-interacting random walk}, also known in statistical mechanics as a polymer measure. In these models, a random walk is reweighted by an interaction term depending on its local time or its increments. A classical question is the existence and qualitative behavior of Gibbs measures relative to a free walk or Brownian motion. Early works such as \cite{Betz2003,BeSp05} established existence criteria and proved central limit theorems for perturbations with sufficiently fast temporal decay. The work~\cite{BrySla95} identified a diffusive phase for weakly self-repelling walks in $d>2$, showing that the perturbation does not destroy Gaussian behavior at large scales. In the context of self-attracting interactions, Donsker--Varadhan~\cite{Donsker1983} computed the free energy in strong-coupling asymptotics for the Fröhlich polaron using their large deviation techniques. In~\cite{Sp87}, the diffusion constant of the reweighted process was heuristically connected to the effective mass of a quantum particle coupled to a scalar quantum field; the follow-up work \cite{DySp20} made this rigorous. More recently, a large body of work has investigated quantitative properties of the diffusion constant, including \cite{BePo23,Se24,BeSchSe25,BeSchSe25b,BaMuSeVa23} via probabilistic methods. Due to the connection discovered in~\cite{Sp87}, the results from \cite{BrSe24,Br24} can also be interpreted as bounds on the diffusion constant.

If the temporal decay of the pair potential $W$ is compactly supported or exponential, the walk remains diffusive: the mean-square displacement satisfies
\[
    \bbE^{\bbPh_{\alpha,T}}[\Vert x_T \Vert^2] \approx \sigma^2 d T,
    \qquad \sigma^2 \neq 1,
\]
with $\sigma^2<1$ in case of an attracting potential and $>1$ in case the potential is repulsive. In this regime, the interaction produces only a finite renormalization of the underlying random walk.
For \emph{attractive} interactions, especially in one or two dimensions, localization (or collapse of the walk) can occur even for short-range potentials. In~\cite{Bolthausen1994} it is proved that a two-dimensional random walk with an attractive potential behaves sub-diffusively. In~\cite{Bolthausen1997}, Bolthausen and Schmock analyzed a general class of $d$-dimensional self-attracting walks; in $d=1$ and under sufficiently strong attraction, the walk clusters in a small region, and $x_T$ converges (when properly rescaled) to a non-trivial random limit. This indicates that in low dimensions there is effectively no finite critical coupling separating diffusive and localized phases: sufficiently strong attraction enforces subdiffusivity or full trapping. This behavior mirrors the classical ``polymer collapse'' phenomenon in statistical mechanics and is conceptually connected to binding in quantum field models such as the Fröhlich polaron~\cite{Spohn1986} in dimension $d \le 2$. We want to note, however, that all techniques that are used to rigorously study behavior of self-interacting random walks discussed so far are based on the large deviation toolbox, meaning that the time decay is either constant or mean--field related.

One interesting feature of our model is the presence of a \emph{long-range temporal interaction}
\[
   W(z,t) = \frac{\Vert z \Vert^\gamma}{|t|^\xi},
\]
which decays slowly in the time separation $|t-s|$ and is not constant. Long-range memory can deeply affect path behavior and is reminiscent of Dyson’s one-dimensional Ising model with power-law interactions~\cite{Dyson1969,Baker1972}, where sufficiently slow decay leads to a genuine phase transition. In hierarchical versions of such spin systems, quantities related to susceptibility satisfy recursive relations across dyadic scales, reflecting the accumulation of correlations over multiple length scales. While our model is not hierarchical by construction, a similar dyadic variance recursion emerges at the level of the analysis, driven by long-range temporal interactions (see Lemma \ref{lemma:Vn_growth}). Rigorous results for long-range \emph{attractive} interactions of Brownian paths were recently obtained by the authors in~\cite{BeSchSe25b}, confirming earlier physical predictions of pinning transitions for long-range self-interacting path measures~\cite{Spohn1986}.

While attractive interactions lead to collapse, the \emph{repulsive} interactions considered in the present work reward expansion of the walk. This parallels the physics of self-avoiding polymers, where avoidance of previously visited sites leads to superdiffusive or even ballistic behavior. In one dimension, even weak repulsion can induce superdiffusivity: Tóth~\cite{Toth1995} proved that the ``true'' self-avoiding walk (TSAW) on $\mathbb{Z}$ scales like $t^{2/3}$ 
and various reinforced or repulsive random walks exhibit similar persistent drift. 
In \cite{PrSaSc08}, the authors study the model defined in equation \eqref{equ:long_range_rw} with $\gamma=2$ and $\xi \in (3,4)$ in two dimensions. They deduce that the process is diffusive for high enough temperature, and ballistic for low enough temperature. This implies that temperature is a much more delicate parameter in the discrete attractive setting compared to the continuous model studied in \cite{BeSchSe25b}. 
More generally, self--repellent walks on $\bbZ$ as in \cite{Toth1995} are defined via a Gibbs measure depending on the local time of the process \cite{AmPaPe83,ToVe11}. The conjectured limiting object for a lot of random self--repellent motions is the ``true`` self--repelling motion, constructed in \cite{ToWe98}. The convergence of TSAW to the true self--repelling motion was verified recently in \cite{KoPe25}. Other self-repelling models include  \cite{NoRoWi87,DuRo92}, where the authors are interested in the solution of an SDE which yields a self--repulsive process. Quantities of interest are the scaling behavior or the mean--square displacement. Results on the mean--square displacement can also be found in \cite{TaToVa12}. The  work \cite{HoToVe10} discusses connections between TSAW and the continuous model. Moreover, scaling limits are provided for both models in $d \ge 3$. In the recent work \cite{CaGi25}, an invariance principle is derived in $d=2$, providing evidence that the mean--square displacement is diffusive times a logarithmic divergence with power $1/2$.

In the present model, we show superdiffusive behavior whenever
\[
   \xi < \frac{\gamma}{2} + c_{\mathrm{crit}}, \qquad c_{\mathrm{crit}}\approx 1.
\]
That is, for sufficiently slow temporal decay of the repulsive potential, the energetic incentive to keep increments large overcomes the entropic cost, and the mean-square displacement grows strictly faster than linearly:
\[
    \bbE^{\bbPh_{\alpha,T}}[ \Vert x_T\Vert ^2] \gg T.
\]
This work follows a general methodology that was recently developed and applied in the continuous polaron model with self-attracting pair potential. The main idea is that one can use correlation inequalities (previously the Gaussian correlation inequality) to compare the full measure $\bbPh_{\alpha,T}$ to a simpler path measure, say $\bbPw_{\alpha,T}$. Two central simplifications are replacing complicated interactions by simpler ones, and omitting interactions between different areas, recovering a Markov property. For more details, see e.g. \cite[Corollary 2.2]{Se24} and Lemma \ref{lemma:omitting_interactions}.
Compared to the previous work~\cite{BeSchSe25b}, we are not able to use the previously developed techniques based on Gaussian domination. While they are powerful and apply to a broad class of problems, they crucially rely on the underlying path measure being Gaussian and the interaction to be self-attracting.
The discrete-time setting and the repulsive long-range interactions studied in this work therefore require a different approach, and the present work extends the range of techniques available for analyzing long-range effects in self-interacting random walks.

\section{Correlation inequalities and change of measure}
The following is the specific GKS inequality that is used for the remainder of this paper. As soon as it is justified that showing \eqref{eq:gks_2_objective} is sufficient to prove the claim, we simply follow the argument in \cite{simon79}.
Let $\caF_0$ be the class of functions with 
$$f \in \caF_0 \iff f: \bbR \rightarrow \bbR, f  \text{ is increasing and non-negative on $(0,\infty)$ and even or odd.}$$
By $\caF$ we denote positive linear combinations of functions in $\caF_0$. Note that $\caF$ is closed under taking products, which is a tremendous advantage compared to increasing functions in the FKG sense. We also define the function class
$$\caF_n := \Big\{ f : \bbR^n \to \bbR, f(x_1,...,x_n) = \prod\limits_{i=1} ^n f_i (x_i), f_i \in \caF \Big\}.$$
In other words, $\caF_n$ lifts the space $\caF$ to $n$ dimensions by taking products of coordinates. 
In particular, p-spin interactions of the form
$F(\phi) = \prod_{k=1} ^n \phi_{t_k} ^{i_k}$
are contained in $\caF_n$. Here, as before,
$$\phi_{t_k} = (\phi_{t_k} ^1,...,\phi_{t_k}^d) \in \bbR^d.$$

\begin{theorem}[{\cite[Theorem 7.1]{simon79}}]
    \label{thm:gks}
    Let $f,g \in \caF_n$ be functions that take only coordinates of one--step increments as input. I.e., 
    $f(x) = f_1 (\phi_1 ^1) f_2 (\phi_1 ^2) \dots f_{dT}(\phi_T ^d)$ with $f_j \in \caF$.
    Then, it holds that
    $$\bbE^{\bbPh_{\alpha,T}}[f g]  \ge \bbE^{\bbPh_{\alpha,T}}[ f] \bbE^{\bbPh_{\alpha,T}}[g].$$
\end{theorem}

\begin{proof}
    By linearity of the integral we may assume that $f_j \in \caF_0$ for all $j \le dT$. Note that 
    $$\bbE^{\bbPh_{\alpha,T}}[f g]  - \bbE^{\bbPh_{\alpha,T}}[ f] \bbE^{\bbPh_{\alpha,T}}[g] = \frac 1 2 \iint \bbPh_{\alpha,T}(\de x)\bbPh_{\alpha,T}(\de y) \big(f(x) - f(y)\big) \big(g(x) - g(y)\big).$$
    Recall that $Z_{\alpha,T}$ denotes the partition function of $\bbPh_{\alpha,T}$. Moreover, take $\Xi(\de \phi) := \prod_{1\le j \le T} \eta (\de \phi_j)$. As a random walk contains the same information as its time steps at each time point, there exists $\Tilde{f}: \phi \rightarrow \bbR$ such that
    $$\Tilde{f}(\phi) = f(x),$$
    if $x$ is the random walk generated by the steps $\phi$. Finally, abbreviating
    $$F(\phi, \varphi) = \big(\Tilde{f}(\phi ) - \Tilde{f}(\varphi)\big) \big(\Tilde{g}(\phi) - \Tilde{g}(\varphi)\big)$$ 
    and
    $$\phi_{[s+1,t]} =\sum\limits_{i=s+1} ^t \phi_i,$$
    it holds that 
    \begin{equation}
        \nn
        \begin{split}
           & 2 Z_{\alpha,T} ^2 \lt( \bbE^{\bbPh_{\alpha,T}}[f g]  - \bbE^{\bbPh_{\alpha,T}}[ f] \bbE^{\bbPh_{\alpha,T}}[g] \rt) \\
           &= \iint \Xi \lt(\de \phi \rt) \Xi \lt(\de \varphi \rt) F(\phi, \varphi)  \exp \lt( \alpha \!\!\!\! \sum\limits_{1 \le s< t \le T} \!\!\! W \lt( \sum\limits_{i = s+1} ^t \phi_i , t-s  \rt) 
           \!+ \!W \lt( \sum\limits_{i = s+1} ^t \varphi _i , t-s  \rt)\!\rt) \\
           &= \iint \Xi \lt(\de \phi \rt) \Xi \lt(\de \varphi \rt) F(\phi, \varphi) \exp \lt( \! \alpha \!\!\ \sum\limits_{s < t } \!\!\!\sum\limits_{i=1} ^{n(t-s)}\!\! c_{i,t-s} \!\!\lt( \!\! \lt( \sum\limits_{p=1} ^d (\phi_{[s+1,t]}^p)^q \rt)^i \!\!\!+ \!\! \lt( \sum\limits_{p=1} ^d (\varphi_{[s+1,t]}^p)^q \!\!\rt)^i \rt)\!\!  \rt). \\
        \end{split}
    \end{equation}
    It thus suffices to show
    \begin{equation}
    \label{eq:gks_2_objective}
    \iint \prod\limits_{i=1} ^m \big( F_i (\phi) + \epsilon_i F_i(\varphi) \big) \Xi( \de \phi) \Xi( \de \varphi) \ge 0
    \end{equation}
    with $F_i$ being a product of $\caF_0$ functions and $\epsilon_i = \pm 1$. This follows from expanding the exponentials into a series (note that products of coordinates of any power are contained in $\caF$). The main point now is that
    $$ab \pm a'b' = \frac 1 2 (a + a')(b \pm b') + \frac 1 2 (a-a')(b \mp b').$$
    That is, \eqref{eq:gks_2_objective} is shown if we can show the claim for $F_i = f_i \in \caF_1$ and $\Xi$ replaced by an even measure $\lambda$ on $\bbR$.
    Since $f_i$ is even or odd,
    $$f_i (\sigma x) = \sigma^{\pi(i)} f_i(x),$$
    with $\sigma = \pm 1$ and  $\pi(i) =2$ if $f_i$ is even and $1$ if $f_i$ is odd. Since $\lambda$ is even we find that \eqref{eq:gks_2_objective} can be written as
    \begin{equation}
        \nn
        \begin{split}
            \eqref{eq:gks_2_objective} &=\sum\limits_{\sigma_1 = \pm 1, \sigma_2 =\pm 1} \int_{x \ge 0, y \ge 0}  \prod\limits_{i=1} ^m \lt( f_i( \sigma_1 x) + \epsilon_i f_i (\sigma_2 y) \rt)  \lambda(\de x)  \lambda(\de y) \\
            &= \sum\limits_{\sigma_3 = \pm 1} \big[ 1+ (-1)^{\sum_{i=1}^m\pi(i)} \big] \int_{x,y \ge 0} \prod\limits_{i=1} ^m  \big( f_i (x) + \epsilon_i \sigma_3 ^{\pi(i)} f_i (y) \big)  \lambda(\de x)  \lambda (\de y). 
        \end{split}
    \end{equation}
In the last equality we consider the cases $\sigma_1 = \sigma_2$ and $\sigma_1 \neq \sigma_2$ with $\sigma_3 = \sigma_1 \sigma_2$. The conclusion is that it suffices to show the claim integrating over $x,y \ge 0$. This is just $1+\prod_{i=1}^m \epsilon_i$
times the integral over $x \ge y \ge 0$. But over this domain all functions we integrate are non-negative (monotone and non-negative by assumption).
\end{proof}

In a first step we use (GKS) to argue that it suffices to treat the case $d=1$. Here, the assumption that the random walk evolves independently in each coordinate is a central ingredient. Indeed, it is possible to treat every coordinate as independent spin, meaning that it is possible to only keep interactions acting on the same dimension.
\begin{corollary}
    \label{cor:1d_suffices}
    Let $\bbP$ be a product random walk with increment distribution $\eta$ and coordinate distribution $\nu$. Denote the one--dimensional reweighted random walk with increments $\nu$ by
    $\bbPh^1 _{\alpha,T}$. Then,
    $$\bbE^{\bbPh_{\alpha,T}}[\Vert x_T \Vert ^2] \ge d \bbE^{\bbPh^1 _{\alpha,T}}[x_T ^2].$$
\end{corollary}

\begin{proof}
   Recall by Assumption \ref{ass:short_range} that
   $$W(z,t) = \sum\limits_{i=1} ^{n(t)} c_{i,t} \lt( \sum\limits_{p=1} ^d (z^p) ^q \rt)^i.$$
   By the binomial Theorem
   $$W(z,t) = \sum\limits_{i=1} ^{n(t)} c_{i,t} \sum\limits_{p=1} ^d   (z^p) ^{qi} + (\text{remainder}),$$
   where it is easily verified that the remainder term is contained in $\caF_n$ for some $n \in \bbN$ due to the fact that it is a (non--negative) linear combination of products of monomials. By (GKS), this term can be omitted, showing the claim.
\end{proof}
Corollary \ref{cor:1d_suffices} allows us to work with random walks in $1$ dimension, which we will do without mentioning it again.

\begin{remark}
    \label{rem:density_replacement}
    We will frequently use Theorem \ref{thm:gks} in the following way: take 
    \begin{equation}
        \label{equ:nn_measure}
        \bbPw_{\alpha,T}(\de x) \propto \exp \lt( \alpha
 \sum\limits_{i=1} ^{T-2} W( x_{i+2} - x_i , 2) \rt) \bbP(\de x).
    \end{equation}
 Since we can expand 
 $$\frac{ \de \bbPh_{\alpha,T} }{\de \bbPw_{\alpha,T}} \propto \exp \lt(\alpha \sum_{\substack{0 \le i < j \le T \\ j-i \neq 2}} W(x_j - x_i , |j-i|)  \rt)$$
 into a series and by interchanging integral and summation, we find that this part of the self-interaction is a uniformly bounded limit of positive linear polynomials of single steps. That is, we find for $f(x) = f_1(\phi_1)f_2 (\phi_2)...$ with $f_j \in \caF$ that
 $$\bbE^{\bbPh_{\alpha,T}}[f] \ge \bbE^{\bbPw_{\alpha,T}}[f].$$
 Of course, there is nothing special about this specific grouping of self-interactions. The point is that one can omit certain self-interaction terms if one only wants to lower bound an average for specific functions $f$.
\end{remark}
The argument provided in the last Remark is formalized in more generality in the following Lemma.
\begin{lemma}
    \label{lemma:omitting_interactions}
    Fix $N \in \bbN$ and let $\phi = (\phi_1,...,\phi_N)$ be i.i.d. spins with symmetric and compact law. For every $I \subseteq [N]$ let
    $H_I$ be a non-negative combination of monomials and 
    $$\phi_I = \sum\limits_{i \in I} \phi_i.$$
    Let $M \subseteq 2^{[N]}$ and define
    $$\bbP_{M}(\de \phi) \propto \exp \lt( \sum\limits_{I \in M} H_I (\phi_I) \rt) \bbP(\de \phi),$$
    where $\bbP$ is the product measure with corresponding single site distribution.
    Then, for any $f \in \caF_N$ taking only $\phi_1,...,\phi_N$ as input and $M \supseteq M'$,
    $$\bbE^{\bbP_M}[f] \ge \bbE^{\bbP_{M'}}[f].$$
\end{lemma}
\begin{proof}
    It is easily seen that 
    $$\frac{\de \bbP_M}{\de \bbP_{M'}}= \frac{1}{\bbE^{\bbP_{M'}}\Big[ \exp \Big( \sum_{I \in M \backslash M'} H_I(\phi_I)\Big)\Big]} \exp \lt( \sum\limits_{I \in M \backslash M'} H_I(\phi_I) \rt).$$
    It therefore suffices to show for $f \in \caF$ as given that
    $$\bbE^{\bbP_{M'}} \lt[f \exp \lt( \sum\limits_{I \in M \backslash M'} H_I(\phi_I) \rt) \rt] \ge \bbE^{\bbP_{M'}} \lt[ f \rt] \bbE^{\bbP_{M'}} \lt[ \exp \lt( \sum\limits_{I \in M \backslash M'} H_I(\phi_I) \rt) \rt].$$
    The result follows by expanding the exponential into a series, exchanging summation and integration and applying Theorem \ref{thm:gks}. The fact that integration and summation can be exchanged follows by Tonelli's Theorem.
\end{proof}

Let us also consider another simplification that we already discussed. As our random walk is assumed to be symmetric and the interaction between any two spins is of product form, we are in the position to apply Wells' inequality. 
\begin{theorem}[Wells' inequality]
    \label{thm:wells_inequality}
    Denote by $\bbPh_{\alpha,\nu,T}$ the measure in \eqref{equ:perturbed_rw} with step distribution $\nu$. Then, there exists $a > 0$ so that uniformly in $T$
    $$\bbE^{\bbPh_{\alpha,\nu,T}} [x_T ^2] \ge \bbE^{\bbPh_{\alpha,\mu,T}} [x_T ^2],$$
    where $\mu = \frac 1 2 ( \delta_a + \delta_{-a})$ is symmetric and supported on $\pm a$. 
\end{theorem}
\begin{proof}
    See the Appendix of \cite{BrLePf81}.
\end{proof}

In other words, it suffices to prove Theorem \ref{thm:short_range} and \ref{thm:long_range_superdiffusive_behavior} for the simple random walk $\bbP$ instead of arbitrary random walks with symmetric and compact step distributions. We will therefore assume now without loss of generality for the remainder of this work that $\bbP$ is the discrete time simple random walk and $\alpha$ is multiplied by $a^i$ for some $i \in \bbN$ as given by Theorem \ref{thm:wells_inequality}. 
Note also that in this special case, the following functions can be treated by the GKS inequality.
\begin{enumerate}
    \item $(x_T)^2 = \Big( \sum\limits_{j=1} ^T \phi_j \Big)^2 = \sum\limits_{1 \le i, j \le T} \phi_i \phi_j$,
    \item $\one_{\phi_i = \phi_j} = \one_{\text{sgn}(\phi_i) = \text{sgn}(\phi_j)} =  \frac 1 2 \lt(1 + \frac{\phi_i}{|\phi_i|}\frac{\phi_j}{|\phi_j|}\rt)$,
    \item $\one_{\phi_1 = \phi_2 = \dots = \phi_N} = \prod\limits_{i=1} ^{N-1} \one_{\phi_i = \phi_{i+1}}$.
\end{enumerate}

\section{Proof of Theorem \ref{thm:short_range}}
We start with a simple calculation. Take $\beta = \zeta \alpha$, where $\zeta > 0$ is the constant taken from Assumption \ref{ass:short_range}.
Moreover, 
we assume that $\e{4\beta} \in \bbN$, which is just for convenience. One could alternatively do the arguments  below with $\e{4\beta}$ replaced by $ \lfloor \e{4\beta} \rfloor$. Finally, we will continue to only work with the simple random walk. As argued earlier, this suffices by Wells' inequality and multiplying $\alpha$ by some fixed constant.
\begin{proposition}
    \label{prop:short_range_prob_exponential_decay}
    Let $0 < j \le T-1$. For $\alpha \ge 1$ it holds that 
    $$\bbPw_{\beta,T} ( \phi_j = \phi_{j+1}) \ge 1- 4\e{-4\beta}.$$
\end{proposition}
\begin{proof}
    By the GKS inequality it suffices to study $\bbP_{2} (\phi_1 = \phi_{2})$, where
    $$\bbP_2 (\de \phi_1 \de \phi_2) \propto \e{ \beta (\phi_1 + \phi_2)^2} \nu(\de \phi_1)\nu(\de \phi_2).$$
    Note that $Z_{2} \ge \frac 1 4 \e{4 \beta}$.
    It therefore holds that
    $$\bbP_2 (\phi_1 \neq \phi_2) \le 4\e{-4\beta},$$
    which implies the claim.
\end{proof}

\begin{proposition}
    \label{prop:short_distance_probability}
    There exists a constant $C_{\ref{prop:short_distance_probability}} >0$ such that, for every $i \in \bbN$ with $i + \e{4\beta}-1 \le T$ and $\alpha \ge 1$, it holds that
    $$\bbPw_{\beta,T}(\phi_i = \phi_{i+1} = ...= \phi_{i+ \e{4\beta}-1}) \ge C_{\ref{prop:short_distance_probability}}.$$
\end{proposition}

\begin{proof}
    By the GKS inequality it suffices to show
    $$\bbPw_{\beta,\e{4\beta}}(\phi_1 = \phi_{2} =...= \phi_{\e{4\beta}}) \ge C_{\ref{prop:short_distance_probability}}$$
    for all $\alpha \ge 1$. Again, by the GKS inequality it holds that
    $$\bbPw_{\beta,\e{4\beta}}(\phi_1 = \phi_{2} =...=\phi_{\e{4\beta}}) \ge \prod\limits_{j=1} ^{\e{4\beta}-1} \bbPw_{\beta,j+1} ( \phi_j = \phi_{j+1}).$$
    Now, use Proposition \ref{prop:short_range_prob_exponential_decay} to see that
    $$\prod\limits_{j=1} ^{\e{4\beta}-1} \bbPw_{\beta,j+1} ( \phi_j = \phi_{j+1}) \ge (1-4\e{-4\beta})^{\e{4\beta}-1}.$$
    For $\alpha \to \infty$ (and hence $\beta \to \infty$) this limit is well-known to be $\e{-4}$, showing the claim.
\end{proof}

\begin{proof}[Proof of Theorem \ref{thm:short_range}]
    W.l.o.g. we take $T$ such that $T / \e{4\beta} \in \bbN$. We additionally choose $\alpha \ge 1 $ so that $\exp(4\beta) \in \bbN$. Then, Theorem \ref{thm:gks} yields the estimate
    $$\bbE^{\bbPw_{\beta,T}} [ x_T ^2] =\bbE^{\bbPw_{\beta,T}} \lt[ \lt( \sum\limits_{n=1} ^T \phi_n \rt) ^2 \rt] \ge \frac{T}{\e{4\beta}} \bbE^{\bbPw_{\beta,\e{4\beta}}} \lt[ \lt( \sum\limits_{n=1} ^{\e{4\beta}} \phi_n \rt) ^2 \rt].$$
    Clearly
    $$\bbE^{\bbPw_{\beta,\e{4\beta}}} \lt[ \lt( \sum\limits_{n=1} ^{\e{4\beta}} \phi_n \rt) ^2 \rt] \ge \bbPw_{\beta,\e{4\beta}}(\phi_1 = \phi_{2} =...= \phi_{\e{4\beta}}) \e{8\beta} \ge C_{\ref{prop:short_distance_probability}} \e{8 \beta}.$$
    Combining yields the desired result.
\end{proof}

\begin{remark}
    Let us briefly discuss how tight the above result is. In order to do this, let us consider the measure $\bbPw_{\beta,T}$ with $W(z,2) = z^2$, which corresponds to the classical one-dimensional Ising model
with nearest-neighbor interaction and zero external field,
\[
H(\sigma) = - \beta \sum_{i=1}^{T-1} \sigma_i \sigma_{i+1},
\qquad \sigma_i \in \{\pm 1\}.
\]
The associated $2\times 2$ transfer matrix is
\[
A=\begin{pmatrix}
e^{\beta } & e^{-\beta } \\
e^{-\beta } & e^{\beta }
\end{pmatrix},
\]
whose eigenvalues are explicitly
\[
\lambda_0 = 2\cosh(\beta ),
\qquad 
\lambda_1 = 2\sinh(\beta ).
\]
Hence,
$\frac{\lambda_1}{\lambda_0} = \tanh(\beta ) =: t$ with
$0 < t < 1$.
In this model the two–point function can be evaluated explicitly to be
\[
\langle \sigma_i \sigma_j\rangle = t^{|i-j|}.
\]
Define the susceptibility
\[
\chi(\beta) 
:= \sum_{k\in\mathbb{Z}} 
\langle \sigma_0 \sigma_k\rangle 
= 1 + 2\sum_{k=1}^\infty t^k
= \frac{1+t}{1-t}.
\]
As $\beta\to\infty$,
\[
t = \tanh(\beta )
= 1 - 2 e^{-2\beta } + o(e^{-2\beta }),
\]
and therefore
\[
\chi(\beta) 
= \frac{1+t}{1-t}
\approx \exp(2\beta ).
\]
Consequently, the double sum of two–point functions satisfies
\[
\sum_{\substack{1\le i,j\le T\\ i\neq j}}
\langle \sigma_i \sigma_j\rangle
= \chi(\beta)\,T + o(T),
\]
and the prefactor grows like $\exp(2\beta )$. Note that similar arguments can be made in case the pair potential has compact support in time (and is non-quadratic) by applying generalized transfer matrix techniques. Therefore,
this reflects that our approach captures the correct first-order behavior of the diffusion constant (note that in our model there is an additional factor of $2$ in front of $\beta$).
\end{remark}

\begin{proof}[Proof of Proposition \ref{prop:odd_interaction}]
By an application of (GKS) it is seen that
$$\bbE^{\bbPh_{\alpha,T}} [\phi_j] \ge \bbE^{\bbP_j} [\phi_j],$$
where
$$\bbP_j (\de \phi_j) \propto \e{ \alpha \phi_j} \nu (\de \phi_j).$$
This simply follows from the fact that the spins satisfy $\phi_k ^2 =1$ and thus every odd degree polynomial taking sums of $\phi_k$ contain at least one factor $\phi_k$.

It is easily seen that $\bbE^{\bbP_j} [\phi_j] = c > 0$.
An application of Jensen's inequality yields
$$\bbE^{\bbPh_{\alpha,T}} [x_T ^2] \ge \lt( \bbE^{\bbPh_{\alpha,T}} [x_T] \rt)^2 \ge T^2  \bbE^{\bbP_j} [\phi_j]^2 \ge c^2 T^2.$$
This shows the claim.
\end{proof}

\section{Recursive variance improvement}
\label{sec:recursive_variance_improvement}

\begin{figure}[t]
    \centering
    \begin{tikzpicture}[>=stealth,scale=0.8]

        \tikzset{
          ball/.style={circle,draw,minimum size=1.3cm,
                       inner sep=0pt,fill=gray!10,font=\small},
          leaf/.style={circle,draw,minimum size=0.6cm,
                       inner sep=0pt,fill=gray!10}
        }

        \def\yroot{0}      
        \def\ylevA{-1.7}   
        \def\ylevB{-3.4}   
        \def\yMid{-5.1}    
        \def\ylevD{-6.8}   

        \node[ball] (root) at (0,\yroot)
        {$\dfrac{1}{\sqrt{T}}x_T$};

        \node[ball] (S1) at (-4,\ylevA) {$\sigma^{1}_{T/2}$};
        \node[ball] (S2) at (4,\ylevA)  {$\sigma^{2}_{T/2}$};

        \draw[->] (root) -- (S1);
        \draw[->] (root) -- (S2);

        \node[ball] (S11) at (-6,\ylevB) {$\sigma^{11}_{T/4}$};
        \node[ball] (S12) at (-2,\ylevB) {$\sigma^{12}_{T/4}$};
        \node[ball] (S21) at ( 2,\ylevB) {$\sigma^{21}_{T/4}$};
        \node[ball] (S22) at ( 6,\ylevB) {$\sigma^{22}_{T/4}$};

        \draw[->] (S1) -- (S11);
        \draw[->] (S1) -- (S12);
        \draw[->] (S2) -- (S21);
        \draw[->] (S2) -- (S22);


        \coordinate (P11) at (-6,\yMid);
        \coordinate (P12) at (-2,\yMid);
        \coordinate (P21) at ( 2,\yMid);
        \coordinate (P22) at ( 6,\yMid);

        \draw[->] (S11) -- (P11);
        \draw[->] (S12) -- (P12);
        \draw[->] (S21) -- (P21);
        \draw[->] (S22) -- (P22);

        \node at (-4,\yMid) {$\cdots$};
        \node at ( 0,\yMid) {$\cdots$};
        \node at ( 4,\yMid) {$\cdots$};

        \foreach \i/\x in {
            1/-7.5,2/-6.5,3/-5.5,4/-4.5,5/-3.5,6/-2.5,7/-1.5,8/-0.5,
            9/ 0.5,10/1.5,11/2.5,12/3.5,13/4.5,14/5.5,15/6.5,16/7.5}
            \node[leaf] (L\i) at (\x,\ylevD) {};

        \draw[->] (P11) -- (L1);
        \draw[->] (P11) -- (L2);
        \draw[->] (P11) -- (L3);
        \draw[->] (P11) -- (L4);

        \draw[->] (P12) -- (L5);
        \draw[->] (P12) -- (L6);
        \draw[->] (P12) -- (L7);
        \draw[->] (P12) -- (L8);

        \draw[->] (P21) -- (L9);
        \draw[->] (P21) -- (L10);
        \draw[->] (P21) -- (L11);
        \draw[->] (P21) -- (L12);

        \draw[->] (P22) -- (L13);
        \draw[->] (P22) -- (L14);
        \draw[->] (P22) -- (L15);
        \draw[->] (P22) -- (L16);

        \draw[red,dashed,thick] (root.south)  -- (0,\ylevD-0.3);

        \draw[red,dashed,thick] (S1.south)    -- (-4,\ylevD-0.3);
        \draw[red,dashed,thick] (S2.south)    -- ( 4,\ylevD-0.3);

        \draw[red,dashed,thick] (S11.south)   -- (-6,\ylevD-0.3);
        \draw[red,dashed,thick] (S12.south)   -- (-2,\ylevD-0.3);
        \draw[red,dashed,thick] (S21.south)   -- ( 2,\ylevD-0.3);
        \draw[red,dashed,thick] (S22.south)   -- ( 6,\ylevD-0.3);

        \node[left] at (-9,\yroot) {$V_n$};
        \node[left] at (-9,\ylevA) {$V_{n-1}$};
        \node[left] at (-9,\ylevB) {$V_{n-2}$};
        \node[left] at (-9,\yMid)  {$\vdots$};
        \node[left] at (-9,\ylevD) {$V_1 = 1$};

    \end{tikzpicture}

    \caption{Dyadic recursion for the normalized endpoint
      \(\frac{1}{\sqrt{T}}x_T\) when \(T = 2^n\). The first levels
      \(\sigma^{1}_{T/2},\sigma^{2}_{T/2}\) and
      \(\sigma^{11}_{T/4},\dots,\sigma^{22}_{T/4}\) are drawn
      explicitly. The horizontal layer of dots indicates further dyadic
      levels before reaching the microscopic blocks at variance \(V_1=1\). Red dashed lines indicate that we can treat variables on separate sides as independent of each other.}
    \label{fig:dyadic-recursion}
\end{figure}
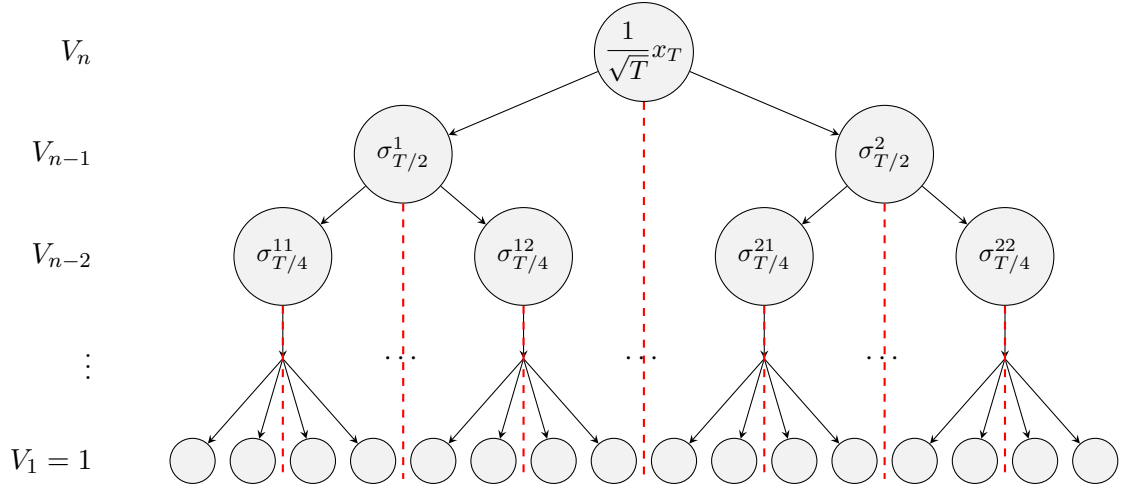

In this Section we prove Theorem~\ref{thm:long_range_superdiffusive_behavior}.  
Throughout we write \(T = 2^t\) with \(t \in \mathbb N\). Let $\xi = \gamma/2 +c$.
Recall that under the tilted measure \(\hat{\bbP}_{\alpha,T,\gamma,\xi}\) from
\eqref{equ:long_range_rw} we are interested in the endpoint variance
\(\bbE^{\hat{\bbP}_{\alpha,T,\gamma,\xi}}[x_T^2]\).
The key observation is that, when \(T\) is dyadic, this quantity can be
controlled recursively in terms of the corresponding variance at time \(T/2\); see also Figure \ref{fig:dyadic-recursion}. For convenience, we will normalize the walk by $T^{-1/2}$, meaning that we study
$$\bbE^{\bbPh_{\alpha,T,\gamma,\xi}}\lt[ \lt( T^{-1/2} x_T  \rt)^2 \rt].$$
This scaling occurs naturally when treating $\xi$ as $\gamma/2 +$(constant).
Split the walk in two halves and normalize them as
\[
    \sigma^1 _{T/2} := \frac{1}{\sqrt{T/2}} \sum_{j=1}^{T/2} \phi_j,
    \qquad
    \sigma^2 _{T/2} := \frac{1}{\sqrt{T/2}} \sum_{j=T/2+1}^{T} \phi_j.
\]
Then, the endpoint \(x_T\) can be written as
\[
    \frac{1}{\sqrt{T}} x_T = \frac{1}{\sqrt{2}} (\sigma^1 _{T/2} + \sigma^2 _{T/2}),
\]
and hence
\[
    \bbE^{\hat{\bbP}_{\alpha,T,\gamma,\xi}}\lt[ \lt( T^{-1/2} x_T  \rt)^2 \rt]
    = \frac 1 2
    \bbE^{\hat{\bbP}_{\alpha,T,\gamma,\xi}}\bigl[(\sigma^1 _{T/2} + \sigma^2 _{T/2})^2\bigr].
\]

The density of \(\hat{\bbP}_{\alpha,T,\gamma,\xi}\) contains three types of
interaction terms:

\begin{enumerate}
    \item interactions internal to the first block \(\{1,\dots,T/2\}\),
    \item interactions internal to the second block \(\{T/2+1,\dots,T\}\),
    \item cross-interactions between the two blocks.
\end{enumerate}

Among the cross-interactions there is a particularly simple and important one,
which is of the form
\begin{equation}
    \label{equ:useful_cross_term}
    \frac{1}{T^c}(\sigma^1 _{T/2} + \sigma^2 _{T/2})^\gamma.
\end{equation}
Here, recall $c = \xi-\gamma/2$.
All other cross-terms are more complicated functions of the increments, but
they only increase the covariance between $\sigma^1 _{T/2}$ and $\sigma^2 _{T/2}$.
By an application of Lemma \ref{lemma:omitting_interactions}, discarding positive cross-terms can only decrease
correlations between the two blocks. In particular, if we retain only the
contribution~\eqref{equ:useful_cross_term}, we obtain a measure
\(\bbP^* _{\beta,\gamma}\) on a pair \((\sigma^1 _{T/2},\sigma^2 _{T/2})\) such that
$$\bbE^{\hat{\bbP}_{\alpha,T,\gamma,\xi}}[x_T ^2]
\ge
\frac 1 2\bbE^{\bbP^*_{\alpha/T^c,\gamma}}[ (y+z) ^2],$$
with
\begin{equation}
    \label{eq:two-block-measure}
    \bbP^* _{\beta,\gamma}(\de y \,\de z)
    \propto
    \exp\!\left(
        \beta\,(y+z)^\gamma
    \right)
    \,\theta(\de y)\,\theta(\de z).
\end{equation}
Here \(\theta\) denotes the law of \(\sigma^1 _{T/2}\) under
\(\hat{\bbP}_{\alpha,T/2,\gamma,\xi}\). In particular, \(\theta\) is symmetric, has second moment $V>0$ and is supported on a finite set.

Therefore, the variance at scale \(T = 2^t\) is bounded
from below by the second moment of the sum of two i.i.d.\ random variables
under an exponential tilt of the form \(\exp(\beta (y+z)^\gamma)\) for some
\(\beta > 0\). By assumption $\gamma \in 2\bbN$, which means by another application of (GKS) we can w.l.o.g. assume $\gamma=2$ (see Lemma \ref{lemma:gamma_two_suffices}).
Expanding the square $(\sigma^1 _{T/2}+\sigma^2 _{T/2})^2$ and by symmetry of
\(\theta\), the marginal terms \((\sigma^1 _{T/2})^2\) and \((\sigma^2 _{T/2})^2\) can once again be omitted in~\eqref{eq:two-block-measure} by invoking Lemma \ref{lemma:omitting_interactions}. The nontrivial
effect of the tilt is therefore to create a positive covariance term \(\sigma^1 _{T/2}\sigma^2 _{T/2}\).
We therefore reduce the problem to
a simpler two-variable tilt
\begin{equation}
    \label{eq:AB-tilt-def}
    \bbP_{\beta,2}(\de y\,\de z)
    \propto
    \exp(\beta yz)\,\theta(\de y)\,\theta(\de z),
    \qquad \beta > 0,
\end{equation}
where \(\theta\) is chosen as before.
In terms of this measure,
\eqref{eq:two-block-measure} implies
\begin{equation}
    \label{eq:variance-from-AB}
    \begin{split}
        2 \bbE^{\hat{\bbP}_{\alpha,T,\gamma,\xi}}\lt[ \lt( T^{-1/2} x_T  \rt)^2 \rt]
    \;&\ge\;
    \bbE^{\bbP_{\beta,2}}[(\sigma^1 _{T/2}+\sigma^2 _{T/2})^2] \\
    \;&=\; \bbE^{\bbP_{\beta,2}}[(\sigma^1 _{T/2})^2] + \bbE^{\bbP_{\beta,2}}[(\sigma^2 _{T/2})^2] + 2\bbE^{\bbP_{\beta,2}}[\sigma^1 _{T/2}\sigma^2 _{T/2}].
    \end{split}
\end{equation}
Symmetry and the construction of \(\bbP_{\beta,2}\) ensure that
\(\bbE^{\bbP_{\beta,2}}[(\sigma^1 _{T/2})^2] = \bbE^{\bbP_{\beta,2}}[(\sigma^2 _{T/2})^2] \ge V\), so the crucial quantity
to bound from below is the covariance term \(\bbE^{\bbP_{\beta,2}}[\sigma^1 _{T/2}\sigma^2 _{T/2}]\).

When \(\beta = 0\) we recover the trivial bound
\(\bbE^{\bbP_{\beta,2}}[\sigma^1 _{T/2}\sigma^2 _{T/2}] = 0\), so~\eqref{eq:variance-from-AB} simply returns the
variance \(V\) at the previous scale.
For \(\beta > 0\) we expect \(\bbE^{\bbP_{\beta,2}}[\sigma^1 _{T/2}\sigma^2 _{T/2}] > 0\) and hence
a genuine increase in variance.
The following Lemma provides a quantitative lower bound that depends only on
the variance \(V\) and not on any further features of \(\theta\).

\begin{lemma}
    \label{lemma:AB-tanh_bound}
    Let \(\theta\) be a symmetric probability measure on \(\bbR\) with
    variance \(V\) supported on a finite set $\Omega$, and let \(\bbP_{\beta,2}\) be the tilted product measure
    defined in~\eqref{eq:AB-tilt-def}. Suppose that \(\beta V \le 2\).
    Then
    \[
        \bbE^{\bbP_{\beta,2}}[yz] \;\ge\; V\,\tanh(\beta V).
    \]
\end{lemma}
As far as we can tell, the condition $\beta V \le 2$ is artificial and only of technical nature. On the other hand, we can assume it to be the case by an application of Lemma \ref{lemma:omitting_interactions}.
The proof of Lemma \ref{lemma:AB-tanh_bound} will be given in Section \ref{sec:convex_geometry_proof}.

Putting all of this together, we find
$$\bbE^{\bbPh_{\alpha,T,\gamma,\xi}}\lt[ \lt( T^{-1/2} x_T  \rt)^2 \rt] \ge  \big( 1+\tanh(\min \{2,\alpha V_{T/2}/T^c\}) \big)V_{T/2}.$$
We can now iterate: by exactly the same argument, it is possible to lower bound $V_{T/2}$ in terms of $V_{T/4}$, corresponding to the random variables $$\sigma^{11}_{T/4} = \sqrt{\frac{4}{T}} \sum\limits_{j=1} ^{T/4} \phi_j,$$
$$\sigma^{12}_{T/4} = \sqrt{\frac{4}{T}} \sum\limits_{j=T/4 +1} ^{T/2} \phi_j$$
and similar definitions for 
$\sigma^{21}_{T/4}$ and $\sigma^{22}_{T/4}$. Writing the recursion in terms of an iteration variable $n \in \bbN$, the question about superdiffusivity can be phrased in terms of whether $V_n$ tends to infinity. Recall
\begin{equation}\label{eq:csat-def}
  c_{\mathrm{crit}} = \log_2\bigl(1+\tanh 2\bigr)\in(0,1).
\end{equation}

\begin{lemma}
    \label{lemma:Vn_growth}
    Fix parameters $\alpha>0$ and $c < c_\mathrm{crit}$, and take the recursion
\begin{equation}\label{eq:Vn-recursion}
  V_1 = 1,\qquad
  V_{n+1} = \bigl(1 + \tanh(\min(\alpha 2^{-cn}V_n,2))\bigr)\,V_n,
  \quad n\ge 1.
\end{equation}
Then, there exists $\alpha_*(c) > 0$ such that
for $\alpha \in (\alpha_* (c),\infty)$, $V_n \to \infty$.

\end{lemma}

\begin{proof}
Fix \( c < c_{\mathrm{crit}} \), and define the rescaled variable
\[
y_n := \alpha\,2^{-cn} V_n.
\]
Then the recursion becomes
\[
y_{n+1} = 2^{-c} \big(1 + \tanh(\min(y_n,2))\big) y_n.
\]
Define
\[
F_c(y) := 2^{-c}(1 + \tanh(\min(y,2)))\,y,
\qquad
g_c(y) := \frac{F_c(y)}{y} = 2^{-c}(1 + \tanh(\min(y,2))).
\]
Then \( y_{n+1} = F_c(y_n) = g_c(y_n)\,y_n \). Since \( \tanh \) is increasing and bounded by 1, \( g_c(y) \) is increasing in \( y \) with
\[
g_c(0^+) = 2^{-c}, \quad g_c(y) \nearrow g_c^{\mathrm{sat}} := 2^{-c}(1 + \tanh 2).
\]
By definition of \( c_{\mathrm{crit}} \), we have \( g_c^{\mathrm{sat}} > 1 \).
Define the critical value
\[
\alpha_*(c) := 2^c \cdot \operatorname{arctanh}(2^c - 1),
\]
which is the unique value such that
\(g_c\big(\alpha_* 2^{-c}\big) = 1.\)
By assumption $\alpha > \alpha_*(c)$, which implies
\( y_1 > \alpha_* 2^{-c} \), so \( g_c(y_1) > 1 \) and \( y_2 > y_1 \). The sequence \( y_n \) increases until it exceeds 2, after which
\[
y_{n+1} = g_c^{\mathrm{sat}} y_n,
\]
with \( g_c^{\mathrm{sat}} > 1 \). So \( y_n = C (g_c ^\mathrm{sat})^n \) for some constant $C>0$, and \( V_n = \alpha^{-1} 2^{cn} y_n \).
\end{proof}

\begin{proposition}
    \label{prop:gamma-square-polynomial-positivity}
    Let $\Phi = \frac{1}{\sqrt{N}} \sum\limits_{j=1} ^N \phi_j$. Then, for $\gamma = 2k$ with $k \in \bbN$ and $N \in \bbN$,
    $\Phi^\gamma - \Phi^2$ is a nonnegative linear combination of even spin monomials of the spins $\phi_j$ ($1 \le j \le N$).
\end{proposition}
\begin{proof}
    Abbreviate $S=\sum_{j=1} ^N \phi_j$.
    Note that $\Phi^\gamma$ is already a sum of functions that we care about. Therefore, it suffices to show that this sum contains $\frac{1}{N} \phi_i\phi_j$ for $i \neq j$.
    This is clear by writing 
    \[
    \Phi^\gamma = \frac{1}{N^{k}} \underbrace{S^2\cdot ... \cdot S^2 } _{k-1 \text{ times}}\sum\limits_{i, j} \phi_i \phi_j.
    \]
    Since $S^2 = N + \sum_{a \neq b}\phi_a\phi_b$, the constant term of $(S^2)^{k-1}$
    is $N^{k-1}$. Hence, when multiplying $S^{2k-2}$ by
    $\sum_{i, j}\phi_i\phi_j$, the coefficient of each $\phi_i\phi_j$ (with $i \neq j$)
    is at least $N^{k-1}$. As all other terms have non-negative coefficients no cancellations can occur. This shows the claim.
\end{proof}

The previous Proposition allows us to switch from general $\gamma$ to $\gamma = 2$. 
\begin{lemma}   
    \label{lemma:gamma_two_suffices}
    Let $\gamma = 2k$ with $k \in \bbN$. Moreover, let $\xi = \gamma/2 + c$ with $c> 0$. Then,
    $$\bbE^{\bbPh_{\alpha,T,\gamma,\xi}} [x_T^2] \ge \bbE^{\bbPh_{\alpha,T,2,1+c}} [x_T ^2].$$
\end{lemma}
\begin{proof}
    Note that the Radon-Nikodym derivative
    $ \de \bbPh_{\alpha,T,\gamma,\xi} / \de \bbPh_{\alpha,T,2,1+c} $ is proportional to the exponential of sums of the form
    $$\beta \lt( \frac{1}{\sqrt{n}} \sum\limits_{j=\ell} ^{\ell+n} \phi_j \rt)^\gamma - \beta  \lt( \frac{1}{\sqrt{n}} \sum\limits_{j=\ell}^ {\ell+n} \phi_j \rt)^2.$$
    Here, $\ell,n \le T$ such that $\ell+n \le T$.
    By Proposition \ref{prop:gamma-square-polynomial-positivity} this difference is the sum of products of $\phi_j$ for $j \in \ell,...,\ell+n$. The result follows from the GKS inequality by using Lemma \ref{lemma:omitting_interactions}.
\end{proof}

The proof of Theorem \ref{thm:long_range_superdiffusive_behavior} follows by putting the previous three Lemmas together and implementing the previously described strategy.

\begin{proof}[Proof of Theorem \ref{thm:long_range_superdiffusive_behavior}]
 By the previous Lemma we find  
    $$\bbE^{\hat{\bbP}_{\alpha,T,\gamma,\xi}}\lt[ \lt( T^{-1/2} x_T  \rt)^2 \rt] \ge \bbE^{\hat{\bbP}_{\alpha,T,2,1+c}}\lt[ \lt( T^{-1/2} x_T  \rt)^2 \rt].$$
    Next, an application of (GKS) yields
    $$\bbE^{\hat{\bbP}_{\alpha,T,\gamma,\xi}}\lt[ \lt( T^{-1/2} x_T  \rt)^2 \rt] \ge  \bbE^{\hat{\bbP}_{\alpha,T/2,2,1+c}} \lt[ (\sigma_{T/2} ^{1})^2 \rt] + \bbE^{\hat{\bbP}_{\alpha,T,2,1+c}} \lt[ \sigma_{T/2} ^{1} \sigma_{T/2} ^{2} \rt].$$
    Take
    $$\bbPw_{\mathrm{split,T}}(\de x) \propto \exp \lt(\alpha \sum\limits_{i < j, j \le T/2} \frac{(x_j - x_i)^2}{|j-i|^{1+c}} + \alpha\sum\limits_{i < j, i \ge T/2+1} \frac{(x_j - x_i)^2}{|j-i|^{1+c}} + \frac{\alpha (\sigma_{T/2} ^2 + \sigma_{T/2} ^1 )^2}{T^c} \rt)\bbP(\de x).$$
    Let $V := \bbE^{\bbPh_{\alpha,T,2,1+c}}[ (\sigma_{T/2} ^{1})^2]$.
    Apply Lemma \ref{lemma:AB-tanh_bound} to find that
    $$\bbE^{\hat{\bbP}_{\alpha,T,2,1+c}} \lt[ \sigma_{T/2} ^{1} \sigma_{T/2} ^{2} \rt] \ge \bbE^{\bbPw_{\mathrm{split,T}}}[\sigma_{T/2} ^{1} \sigma_{T/2} ^{2}] \ge V \tanh \lt(\frac{\alpha}{T^c}V \rt).$$
    In order to calculate $V$, we again write
    $$\bbE^{\bbPh_{\alpha,T/2,2,1+c}}[(\sigma_{T/2} ^{1})^2] \ge \bbE^{\bbPh_{\alpha,T/4,2,1+c}} [(\sigma_{T/4} ^{11})^2]  + \bbE^{\bbPh_{\alpha,T/2,2,1+c}}[(\sigma_{T/4} ^{11})(\sigma_{T/4} ^{12})].$$
    Similar to before, we again apply Lemma \ref{lemma:AB-tanh_bound} to the measure which treats both r.v. as independent with an exponential tilt. Recursing this is exactly what is studied in Lemma \ref{lemma:Vn_growth}, yielding the result.
\end{proof}

\section{Symmetric laws under a one–variable tilt}
\label{sec:convex_geometry_proof}
The goal of this Section is to prove Lemma \ref{lemma:AB-tanh_bound}.
Recall that in the recursive scheme of Section~\ref{sec:recursive_variance_improvement},
we encountered a \emph{symmetric} random variable
$S = xy$ with variance $V^2$ (our random variables were symmetric i.i.d. with variance $V$)
and we wish to understand the behavior of the expectation
\[
    \mathbb{E}^{\bbP_\beta}[S],\qquad \beta>0,
\]
under the exponentially tilted probability measure
\begin{equation}\label{eq:S-tilt-law}
    \mathbb{P}_{\beta}(\mathrm{d}s)
    \propto 
    e^{\beta s}\,\mathbb{P}(S\in\mathrm{d}s).
\end{equation}
Our goal is to find a \emph{universal lower bound} on $\mathbb{E}^{\bbP_\beta}[S]$ 
that depends only on the variance $V^{2}$ and holds for \emph{every} symmetric law of variance $V^{2}$.

Even if $S$ is supported on a finite, symmetric set $\Omega$, this problem is high–dimensional in principle, since $\Omega$ might be large and the law of $S$ arbitrary.  
We now show that the extremal structure of the feasible set allows one to reduce the analysis to a \emph{four–point} measure, i.e. a measure with $4$ atoms at most. 

In all that follows, we work on the space of probability measures $\caP_1 (\Omega)$ on the discrete, finite set $\Omega$ so that $x \in \Omega \iff -x \in \Omega$. Next, take 
\[
    \mathcal{S}_{V^{2}}
    :=
    \Bigl\{
        \mu \in \caP_1 (\Omega) \text{  is symmetric and}
        \int x^{2}\,\mu(\mathrm{d}x) = V^{2}
    \Bigr\}.
\]
This is a weakly compact and convex set.
For $\mu\in\mathcal{S}_{V^{2}}$ define the affine functionals
\[
    L_{1}(\mu) := \int x\, e^{\beta x}\,\mu(\mathrm{d}x),\qquad
    L_{2}(\mu) := \int e^{\beta x}\,\mu(\mathrm{d}x),
\]
so that
\[
    \mathbb{E}^{\mu_\beta}[S]
    = \frac{L_{1}(\mu)}{L_{2}(\mu)}.
\]
The following Lemma is classical. We provide a proof for the reader's convenience.

\begin{lemma}\label{lem:ratio-extreme-attained}
Let $K$ be a nonempty compact convex subset of a locally convex topological
vector space, and let $L_1,L_2:K\to\mathbb{R}$ be continuous affine maps with
$L_2(x)>0$ for all $x\in K$. Set
\[
    F(x):=\frac{L_1(x)}{L_2(x)},\qquad x\in K.
\]
Then $F$ attains its minimum on $K$ at an extreme point of $K$.
\end{lemma}

\begin{proof}
Since $L_1$ and $L_2$ are continuous on the compact set $K$ and $L_2>0$ on $K$,
the functional $F$ is continuous on $K$. Hence there exists $x_0\in K$ such that
\[
    F(x_0) = \min_{x\in K} F(x) =: m.
\]

We first show that if $x_0$ is a minimiser and can be written as a nontrivial
convex combination $x_0 = t x_1 + (1-t)x_2$ with $x_1,x_2\in K$, $t\in(0,1)$,
then both $x_1$ and $x_2$ are also minimisers, i.e.\ $F(x_1)=F(x_2)=m$.

Because $L_1$ and $L_2$ are affine, we have
\[
    L_i(x_0)
    = t L_i(x_1) + (1-t)L_i(x_2),\qquad i=1,2.
\]
Since $F(x_0)=m$, this is equivalent to
\[
    L_1(x_0) = m L_2(x_0).
\]
Suppose for contradiction that $F(x_1)>m$. Then there exists $\varepsilon>0$
such that
\[
    F(x_1) \ge m + \varepsilon
    \quad\Longrightarrow\quad
    L_1(x_1) \ge (m+\varepsilon) L_2(x_1).
\]
On the other hand, by minimality of $x_0$ we have $F(x_2)\ge m$, hence
\[
    L_1(x_2) \ge m L_2(x_2).
\]
Combining these,
\begin{align*}
    L_1(x_0)
    &= t L_1(x_1) + (1-t)L_1(x_2)
    \\
    &\ge t (m+\varepsilon)L_2(x_1) + (1-t) m L_2(x_2)
    \\
    &= m\bigl(t L_2(x_1)+(1-t)L_2(x_2)\bigr)
       + t\varepsilon L_2(x_1)
    \\
    &= m L_2(x_0) + t\varepsilon L_2(x_1).
\end{align*}
Since $L_2(x_1)>0$, the last term is strictly positive, so we obtain
\[
    L_1(x_0) > m L_2(x_0),
\]
which contradicts $L_1(x_0) = m L_2(x_0)$. Thus $F(x_1)\le m$. Together with
$F(x_1)\ge m$ (by minimality of $m$), we get $F(x_1)=m$. The same argument with
$x_1$ and $x_2$ interchanged shows $F(x_2)=m$ as well.

Now let
\[
    M := \{x\in K : F(x)=m\}
\]
denote the set of minimisers. Note that $M$ is convex by the fact that $L_1$ and $L_2$ are affine. The set $M$ is nonempty (it contains
$x_0$) and compact (as a closed subset of $K$).

By the Krein--Milman theorem, $M$ has at least one extreme point $x^\ast$,
and every extreme point of $M$ is also an extreme point of $K$. 
Finally, let $x^*$ be an extreme point of $M$ such that
$x^* = t y + (1-t) z$ for $y,z \in K$ and $t \in (0,1)$. By what we have just shown, $y,z \in M$. By assumption $x^*$ is an extreme point, meaning that $y = z = x^*$. Therefore, an extreme point of $M$ is an extreme point of $K$. This shows the claim.
\end{proof}
Specializing the previous Lemma to our situation, we find that
$\min_{\mu \in \caS_{V^2}} \bbE^{\mu_\beta}[S]$
is attained at an extreme point of $\caS_{V^2}$. Note that $\caS_{V^2}$ can be identified with the probability simplex over $N+1$ dimensions, where $2N +1 := |\Omega|$ and two linear constraints, i.e.
$$\caS_{V^2} \simeq \lt\{ w \in \bbR^{N+1}: w_j \ge0, \sum\limits_{i=0} ^{N} w_i = 1, \sum\limits_{i=1} ^N x_i ^2 w_i = V^2 \rt\}.$$
Here, $w = (w_0,\dots,w_N)$ corresponds to the probability measure
$w_0 \delta_0 + \sum\limits_{i=1} ^N \frac {w_i} 2 (\delta_{-x_i} + \delta_{x_i}).$

\begin{proposition}
    The extreme points of $\caS_{V^2}$ have at most four atoms. I.e., at most two weights $w_j$ are greater than zero.
\end{proposition}

\begin{proof}
    Let $w \in \bbR^{N+1}$ and assume that there exists three coordinates (which we take to be $1,2,3$ for simplicity) so that $w_j > 0$ for $0 < j \le 3$. We show that $w$ is the midpoint of two other weights vectors $w^+,w^-$ in $\caS_{V^2}$. Indeed, the linear system of two equations
    $$ w_1 + w_2 + w_3 = 1 - w_0 -  \sum\limits_{i=4} ^N w_i$$
    $$w_1 x_{1} ^2 + w_2 x_2 ^2 + w_3 x_3 ^2 = V^2 - \sum\limits_{i=4} ^N w_i x_i ^2$$
    can be interpreted as a system of three unknowns $w_1,w_2$ and $w_3$. Therefore, it is easily verified that the solution manifold is a line in $\bbR^3$. Since $w_j > 0$ for $0<j\le 3$, there exists a direction $h \in \bbR^3$ and $\epsilon > 0$ such that
     $$w^{\pm} = (w_1 \pm \epsilon h_1 ,w_2 \pm \epsilon h_2,w _3 \pm \epsilon h_3, w_4,...,w_N) \in \caS_{V^2}.$$
     In particular, $w = (w^+ + w^-)/2$, which by definition means that $w$ is no extreme point.
\end{proof}

\begin{remark}
    While we do not need a more general statement, we want to note that this proof can possibly be extended to the setting in which $\Omega$ is at least a compact subset of e.g. $\bbR$. The fact that $n$ affine conditions on a simplex of probability measures reduce the extreme points to (at most) $n+1$ atomic measures is a classical result \cite{Wi88}.
\end{remark}

Our conclusion is that an extreme point of $\caS_{V^2}$ is at most four atomic, where the atoms come either in pairs or the minimizer is three-atomic with an atom at $0$. Let $S$ be a random variable with such a law, i.e.
\[
    S\in\{\pm a,\pm b\}.
\]
A short computation yields
\begin{align}
    Z(a,b,p;\beta)
    &:= \mathbb{E}\bigl[e^{\beta S}\bigr]
      = p\cosh(\beta a) + (1-p)\cosh(\beta b), \label{eq:Z-four-point}\\
    N(a,b,p;\beta)
    &:= \mathbb{E}\bigl[S\,e^{\beta S}\bigr]
      = p\,a\,\sinh(\beta a)
        + (1-p)\,b\,\sinh(\beta b).\label{eq:N-four-point}
\end{align}
Therefore
\begin{equation}\label{eq:S-tilted-ratio-four-point}
    \mathbb{E}^{\bbP_\beta}[S]
    = \frac{N(a,b,p;\beta)}{Z(a,b,p;\beta)}.
\end{equation}

Therefore, for fixed $V>0$ and $\beta>0$,
\begin{equation}
    \label{equ:two-dim_optimization_problem}
    \inf_{\mu\in\mathcal{S}_{V^{2}}}
        \mathbb{E}^{\mu_\beta}[S]
    =
    \inf_{\substack{0<a\le b,\; p\in[0,1]\\
                    pa^{2}+(1-p)b^{2}=V^{2}}}
        \frac{N(a,b,p;\beta)}{Z(a,b,p;\beta)}.
\end{equation}
This problem is still non-trivial to solve. In order to make progress, we now rely on the assumption $\beta V \le 2$.

\begin{proposition}
    \label{prop:minimiser_is_two_point}
    The minimizer of the problem in \eqref{equ:two-dim_optimization_problem} is given by the symmetric measure
    $$\mu = \frac 1 2 (\delta_V + \delta_{-V}).$$
    In other words, the optimal measure is two-atomic with atoms at $\pm V$.
\end{proposition}

\begin{proof}
    Let $S$ be concentrated on the four points $\{\pm a, \pm b\}$ so that its variance is $V^2$ and 
    $F_{\beta}(a,b,p) :=\bbE^{\bbP_\beta}[S]$ is given by the fraction in \eqref{eq:S-tilted-ratio-four-point}. Define the function
    $$k(x) := x \sinh(\beta x) - V\tanh(V \beta) \cosh(\beta x).$$
    A simple calculation shows that
    $$F_\beta (a,b,p) \ge V \tanh(V \beta) \iff \caN(a,b,p):= pk(a) + (1-p)k(b) \ge 0.$$
    Our goal will be to show that $\varphi(t) := k(\sqrt{t})$ is convex on $[0,\infty)$, which will yield the claim. Differentiate $k$ two times to find
    \begin{align*}
k'(x)
&= (1 - \beta V\tanh(\beta V))\sinh(\beta x) 
   + \beta x\cosh(\beta x), \\
k''(x)
&= \beta\cosh(\beta x)\bigl(2 - \beta V\tanh(\beta V)\bigr)
   + \beta^2 x\sinh(\beta x).
\end{align*}

Since $\beta V<2$ implies $\beta V\tanh(\beta V)<2$, we have
\[
2 - \beta V\tanh(\beta V) > 0,
\]
so $k''(x)>0$ for all $x>0$, showing convexity.
Next, compute
\[
\varphi''(t)
= \frac{1}{4 t^{3/2}}\bigl(xk''(x)-k'(x)\bigr),
\qquad x=\sqrt t
\]
and take $h(x):=xk''(x)-k'(x)$. Differentiating once we find
$h'(x)=xk'''(x)$.
Since
\[
k''(x)
= \beta\cosh(\beta x)\bigl(2 - \beta V\tanh(\beta V)\bigr)
  + \beta^2 x\sinh(\beta x),
\]
one checks
\[
k'''(x)
= \beta^2\sinh(\beta x)\bigl(3 - \beta V\tanh(\beta V)\bigr)
  + \beta^3 x\cosh(\beta x),
\]
which is strictly positive for all $x>0$ because  
$\beta V\tanh(\beta V) < 2$ implies  
$3 - \beta V\tanh(\beta V) > 1$.
Thus $h'(x)>0$, and since $h(0)=0$
\[
h(x)>0 \quad\text{for all }x>0.
\]
Hence $\varphi''(t)>0$ for all $t>0$, i.e.\ $\varphi$ is strictly convex.
Let $Y = S^2$, which means that
\[
\mathbb{P}(Y=a^2)=p,\qquad \mathbb{P}(Y=b^2)=1-p.
\]
By the variance constraint on $S$ it holds that
$\mathbb{E}[Y] = p a^2 + (1-p)b^2 = V^2$.
Note that
\[
\mathcal{N}(a,b,p)
= \mathbb{E}[k(\sqrt Y)]
= \mathbb{E}[\varphi(Y)].
\]
Since $\varphi$ is strictly convex, an application of Jensen's inequality yields
\[
\mathbb{E}[\varphi(Y)]
\ge \varphi(\mathbb{E}[Y])
= \varphi(V^2)
= k(V)
= 0.
\]
Strictness holds unless $Y$ is a.s.\ constant, i.e.\ $a^2=b^2=V^2$.
Therefore $\mathcal{N}(a,b,p)\ge 0$, with equality only in the two--point case.  
Equivalently,
\[
F_\beta(a,b,p) \ge V\tanh(\beta V),
\]
with equality if and only if $a=b=V$.
\end{proof}

\begin{proof}[Proof of Lemma \ref{lemma:AB-tanh_bound}]
    Clearly
    $$\bbE^{\bbP_{\beta,2}}[yz] \ge \inf\limits_{\theta \in \caS_{V^2}} \bbE^{\bbP_{\beta,\theta}} [yz],$$
    where
    $$\bbP_{\beta,\theta}(\de y \de z) \propto \exp \lt( \beta yz  \rt) \theta (\de y) \theta ( \de z).$$
    By Proposition \ref{prop:minimiser_is_two_point} we find that the infimum on the rhs is obtained by taking $\theta = (\delta_V + \delta_{-V})/2$. The claim follows by the identity \eqref{eq:S-tilted-ratio-four-point}.
\end{proof}

{\bf Acknowledgments:} This research was supported by DFG grant No 535662048.
\bibliographystyle{alpha}
\bibliography{bib}
\end{document}